\documentclass[pdftex,a4paper,11pt]{amsart}

\usepackage{amsfonts, amsthm, amsmath, cite, esint}
\usepackage{mathrsfs}
\usepackage[english]{babel}
\usepackage[pdftex]{color,graphicx}

\theoremstyle{plain}
  \newtheorem{theorem}{Theorem}
  \newtheorem{proposition}[theorem]{Proposition}

\theoremstyle{definition}
  
\theoremstyle{remark}

\newcommand{\comma}{\textrm{,}}
\newcommand{\period}{\textrm{.}}
\newcommand{\bra}[1]{\left( #1 \right)}
\newcommand{\sqa}[1]{\left[ #1 \right]}
\newcommand{\cur}[1]{\left\{ #1 \right\}}
\newcommand{\ang}[1]{\left< #1 \right>}
\newcommand{\abs}[1]{\left| #1 \right|}
\newcommand{\nor}[1]{\left\| #1 \right\|}

\newcommand{\rnum}{\mathbb{R}}
\newcommand{\nnum}{\mathbb{N}}

\newcommand{\eE}{\mathbb{E}}

\newcommand{\divex}[1]{\mathrm{div}_x\bra{#1}}
\newcommand{\dive}{\mathrm{div}}

\newcommand{\divey}[1]{\mathrm{div}_y\bra{#1}}

\newcommand{\veps}{\varepsilon}
\newcommand{\xe}{x_\veps}

\newcommand{\ye}{y_\veps}

\newcommand{\cL}{\mathcal{L}}

\newcommand{\cF}{\mathcal{F}}

\begin{document}


\author{Dario Trevisan}
\address{Scuola Normale Superiore, Piazza dei Cavalieri, 7, 56126, Pisa, Italy}
\email[D.~Trevisan]{dario.trevisan@sns.it}

\title{Lagrangian flows driven by $BV$ fields in Wiener spaces}


\begin{abstract}
We establish the renormalization property for essentially bounded solutions of the continuity equation associated to $BV$ fields in Wiener spaces, with values in the associated Cameron-Martin space; thus obtaining, by standard arguments, new uniqueness and stability results for correspondent Lagrangian $L^\infty$-flows. An example related to Neumann elliptic problems is also discussed.
\end{abstract}

\maketitle

\section{Introduction}


The finite dimensional theory of flows driven by weakly differentiable fields began with the work by R.J.~DiPerna and P.-L.~Lions, \cite{MR1022305}, where it is proved that Sobolev regularity for vector fields in $\rnum^n$ is sufficient to obtain existence and uniqueness of a generalized notion of flow driven by these fields. Since then it has found many developments and applications; for brevity, we refer to the exposition in \cite{MR2408257} mentioning only some important recent papers: \cite{MR2849722}, \cite{MR2369485}, \cite{MR2450159} and \cite{MR2375067}. The work of DiPerna and Lions did not cover the case of bounded variation ($BV$) fields, which arise naturally in many contexts: it was settled by L.~Ambrosio in \cite{MR2096794} and since then, $BV$ fields were considered in other settings, e.g.\ SDEs, in \cite{MR2891455}, or Fokker-Planck equations, in \cite{MR3016531}.

On abstract Wiener spaces, a theory of flows was developed somehow independently from that of DiPerna and Lions (see \cite{MR724704}, \cite{MR724705} and \cite{MR759104} for early works and \cite{MR2199294} for recent developments). In \cite{MR2475421}, tools from DiPerna-Lions theory were be applied also in the infinite-dimensional setting, obtaining existence and uniqueness of flows driven by Sobolev fields. The natural extension to Fokker-Planck equations was developed in \cite{MR2669050}. A comparison between the two approaches seems difficult in general, due to different assumptions on different norms: one approach might be better than the other, depending on the nature of the driving field (for more details, see the introduction of \cite{MR2475421}).

\subsection*{Aim of the article}
In Wiener spaces, uniqueness for flows driven by $BV$ fields was left open in \cite{MR2475421} and the aim of this article is to settle it. The motivation to deal with $BV$ fields is mainly due to the lack of a local analogue of the results for Sobolev fields obtained in \cite{MR2475421}: in Section \ref{section-example} below, we show in a concrete situation how a global $BV$ field arise naturally from the solution of an elliptic problem in a domain. If extension theorems for Sobolev classes on infinite dimensional domains were known, one might be able to work without $BV$ fields, but this is a rather delicate subject (see \cite{bogachev-extensions} for some recent results, in a negative direction).

Some reasons for the existence of this gap between the finite dimensional theory and the Wiener space theory can be traced in the fact that the theory of $BV$ maps on Wiener spaces (which began with the works by Fukushima and Hino in \cite{MR1761369} and \cite{MR1837539}) only recently has been studied from a geometric point of view, closer to the finite dimensional setting: see \cite{MR2558177}, \cite{MR2847889} and \cite{MR3034584}. 

For a presentation of the general problem of flows and the ideas involved in the Wiener space setting, we refer to the  well written introduction of \cite{MR2475421} and then to Section 5 therein for a rigorous derivation of the links between well-posedness (i.e.~existence and uniqueness) of flows driven by a field $b$ and that of the associated continuity equation,
\begin{equation}\label{eq-continuity} \frac{d}{dt} u_t + \dive\bra{ b_t u_t} = 0 \comma \end{equation}
where $\dive = \dive_\gamma$ denotes the divergence with respect to the underlying Gaussian measure.

While existence is settled rather easily, assuming bounds on $b$ and its distributional divergence $\dive b$, uniqueness is  a difficult issue, already in the finite dimensional setting. The DiPerna-Lions argument is based on the notion of renormalized solution, whose definition we recall here: a solution of \eqref{eq-continuity} is said to be renormalized if, for every $\beta \in C^1\bra{\rnum}$, with both $\beta'\bra{z}$ and $\beta\bra{z}-z\beta'\bra{z}$ bounded, it holds, in the distributional sense,
\begin{equation} \label{eq-def-renormalization} \frac{d}{dt}  \beta\bra{u_t}  + \dive \bra{ b_t \beta\bra{u_t} } - \dive b_t\sqa{\beta\bra{u_t}-u_t\beta'\bra{u_t}}  =0  \period \end{equation}
If all the solutions (in a certain space) are known to be renormalized, it is well-known that uniqueness holds in that space (for a precise statement, see e.g.\ the proof of Theorem 3.1 at the end of Section 3 in \cite{MR2475421}). 

In this article, therefore, we focus on the proof of the renormalization property, the main result being Theorem \ref{main-theorem} below: given a $BV$ vector field $b$ with integrable divergence, every solution of \eqref{eq-continuity} in $L^\infty\bra{(0,T)\times X}$ is renormalized: from this it is not difficult to recover statements about uniqueness and stability for $L^\infty$ Lagrangian flows as in \cite{MR2475421}.

\subsection*{Comments on the proof technique and structure of the paper}

In the finite dimensional setting, the approach developed in \cite{MR2096794} is based on a refined analysis of error terms arising from smooth approximations, with two different estimates: an anisotropic estimate, which is rather good in the regions where the measure-derivative $Db$ is mostly singular with respect to the Lebesgue measure, and an isotropic estimate, which is good instead in the regions where the derivative is mostly absolutely continuous. Then, an optimization procedure on the choice of approximations gives the renormalization property.

In the Wiener setting, a direct implementation of this method fails, because of error terms depending on the dimension of the space. Our contribution may be summarized in obtaining a refined anisotropic estimate which is well-behaved at every point and, after an optimization procedure, turns out to be sufficient to conclude the renormalization property.

This method works also in the finite dimensional setting and since the steps might prepare the reader for the Wiener case, we describe it briefly in Section \ref{finite-dimensional}.

Starting from Section \ref{section-wiener}, we deal uniquely with the Wiener setting: we recall some definitions and facts about Sobolev and $BV$ maps and then, in Section \ref{section-main-result} we state the main result, Theorem \ref{main-theorem}. In Section \ref{section-example} we discuss an example which we believe motivates the necessity for dealing with $BV$ fields. Section \ref{section-technical} is devoted to establish some technical facts that are instrumental to the proof of Theorem \ref{main-theorem}, which is finally discussed in Section \ref{section-proof}.

\subsection*{Acknowledgements}

We thank L.~Ambrosio and M.~Novaga for many discussions and suggestions on the subject and valuable comments that improved the manuscript. We warmly thank G.~Da Prato for helping us with references in Section \ref{section-example}.

\section{Renormalized solutions in $\rnum^d$}\label{finite-dimensional}

The aim of this section is to prove the renormalization property for $L^\infty$ solutions of the continuity equation associated to a finite dimensional $BV$ field, along the same lines as in Section \ref{section-proof} below: there, computations are a bit more involved and thus the hope is that this section might guide the reader towards Section \ref{section-proof}. In particular, the global structure of the proof is exactly the same.

This provides also an alternative proof of Theorem 3.5 in \cite{MR2096794}: for simplicity, here we work under global assumptions, but the arguments can be easily adapted to cover the $BV_{loc}$ case. For brevity, we completely refer to Section 5 in \cite{MR2408257}, for a detailed introduction of the finite dimensional setting. Recall that the class of test functions is given by $C_c^\infty ((0,T) \times \rnum^d)$.

\begin{theorem}\label{theorem-1}
Let $b \in L^1\bra{ (0,T); BV(\rnum^d; \rnum^d)}$, $\divex{ b_t } \in L^1\bra{ (0,T) \times \rnum^d}$. Any distributional solution $u = \bra{u_t} \in L^\infty\bra{(0,T)\times \rnum^d}$ of
\[ \frac{d}{dt} u_t + \divex{b_t u_t} = 0 \]
is a renormalized solution.
\end{theorem}

\begin{proof}
We enumerate the essential steps to prove the result, with the convention that step $n$ is then developed with more details details in the correspondent subsection below. 

\begin{description}

\item[1. Mollification] We set up a two parameter family (with parameters $\rho$ varying in some set and $\veps>0$) of mollified solutions $u^\veps_\rho$ that solve in the distributional sense, the equation
\begin{equation}\label{eqremainder} \frac{d}{dt} (u_\rho^\veps)_t + \dive\bra{b (u_\rho^\veps)_t} =  (r_\rho^\veps)_t \period \end{equation}
For simplicity, we omit in what follows the dependence on $\rho$.
\item[2. Approximate renormalization] We prove that $u^\veps$ and all the terms above are sufficiently smooth so that, given any $\beta \in C^1\bra{\rnum}$, with both $\beta'\bra{s}$ and $\beta\bra{s} - s\beta'\bra{s}$ uniformly bounded, standard calculus can be applied to deduce that
\begin{equation}\label{eq-mollified-renormalized}
 \frac{d}{dt} \beta\bra{u^\veps_t} + \divex{b \beta\bra{u^\veps_t}} - \sqa{\beta\bra{u^\veps_t} - \beta'\bra{u^\veps_t} u^\veps_t} \divex b =  \beta'\bra{u^\veps_t} r^\veps \period \end{equation}
\item[3. Anisotropic estimate] We prove that there exists some function $\Lambda_\rho\bra{t,x}$ such that, for every test function $\varphi$, it holds
\begin{equation} \label{anisotropic} \limsup_{\veps \to 0} \int_0^T\int_{\rnum^d} \abs{\varphi \beta'\bra{u^\veps_t} r^\veps} dt dx \le \nor{u}_\infty \nor{\beta'}_\infty \int_{(0,T)\times \rnum^d} \abs{\varphi} \Lambda_{\rho}\, d\abs{Db} \period \end{equation}
where $\abs{Db} = \abs{Db_t} dt$ defines a finite Borel measure on $(0,T)\otimes \rnum^d$. This entails that the distribution
\[ \frac{d}{dt} \beta\bra{u_t} + \dive \bra{b \beta\bra{u_t}} - \sqa{\beta\bra{u_t} - \beta'\bra{u_t} u_t} \dive\bra{b} = \sigma\]
is a real valued measure with total variation smaller than $\Lambda_\rho \abs{Db}$, the so-called defect measure.
\item[4. Optimization] We show that
\[ \abs{\sigma} \le \bigwedge_\rho \Lambda_\rho \abs{Db} = 0\]
which settles the renormalization property.
\end{description}
\end{proof}

\subsection{Mollification}
Let $\rho$ be any smooth function defined on $\rnum^d$, with compact support and $\int \rho = 1$. Given $\veps >0$, and $f \in L^1_{loc} \bra{(0,T)\times\rnum^d}$, we mollify by convolution along the space variables, defining
\[ T^\veps_\rho \varphi \bra{t,x} =  \int_{\rnum^d} \varphi\bra{t,x_\veps} \rho\bra{y} dy,\]
where we write here and in all the subsections below, $x_\veps = x- \veps y$. The adjoint operator (in $L^2((0,T)\times \rnum^d)$) is given by the expression
\[(T^\veps_\rho)^* \varphi\bra{t,x} = \int_{\rnum^d} \varphi\bra{t, x^\veps} \rho\bra{y} dy \comma \]
where we write here and in all the subsections below, $x^\veps = x+ \veps y$. These operators preserve test functions and so we also define $(T^\veps_\rho)^*$ on distributions, by
\[ \ang{\varphi, \bra{T^\veps_{\rho}}^* L} = \ang{T^\veps_{\rho}\varphi, L} \period\]
We let finally $u_\rho^\veps = (T_\rho^\veps)^*u$, so that it holds
\[ (r_\rho^\veps)_t = \dive\bra{ b (u_\rho^\veps)_t} - (T_\rho^\veps)^*\dive\bra{b u_t}  \period\]

\subsection{Approximate renormalization}\label{approximate-renormalization}
To keep notation simple, we frequently omit here and below the dependence on $t$ and $\rho$, since they play no role.

The function $u^\veps$ is smooth with respect to the space variables, so that the only thing to be proved to justify the computations stated above, is that the distribution $r^\veps$ is (induced by) an integrable function. It is enough to prove that both $\dive\bra{b u^\veps}$ and $\bra{T^\veps}^*\dive\bra{b u}$ are integrable functions.

\subsubsection{Equivalent expressions for $\dive\bra{b u^\veps}$ and $\bra{T^\veps}^*\dive\bra{b u}$ via integration by parts}

Let $\varphi$ be a test function and compute
\[ \int \varphi \,\dive\bra{b u^\veps} = -  \int \ang{\nabla \varphi, b} u^\veps = -  \int  u \,T^\veps \ang{\nabla \varphi, b} \period\]
Integrating by parts, we obtain 
\[  T^\veps\ang {\nabla \varphi , b}\bra{x} = \frac 1 \veps \int \varphi\bra{\xe} \divey{ b\bra{\xe} \rho\bra{y}} dy = \int \varphi\bra{\xe} A_\veps\bra{x,y} \period \]
If we prove that $A_{\veps} \in L^{1}\bra{\rnum^{2d}}$,  with the change of variables $(x,y) \mapsto (x^\veps, y)$ we get
\[ \int \varphi\, \dive\bra{b u^\veps} = - \int \varphi \bra{x} \int u^\veps\bra{x^\veps} A_{\veps} \bra{x^\veps, y} dy \period\]
which gives $\dive\bra{b u^\veps} \in L^1$. For $\bra{T^\veps}^*\dive\bra{b u}$ we proceed similarly,
\[ \int \varphi \bra{T^\veps}^*\dive\bra{b u}  = - \int u \ang{ \nabla T^\veps \varphi, b }\period \]
Then, integrating by parts,
\[ \ang{\nabla T^\veps \varphi , b} \bra{x} = \frac 1 \veps \int \varphi\bra{\xe}  \divey{ b\bra{x} \rho\bra{y}} dy \period =  \int \varphi\bra{\xe} B_\veps \bra{x, y} dy \period \]
If we prove $B_\veps \in L^{1}\bra{ \rnum^{2d}}$, we conclude that $\bra{T^\veps}^*\dive\bra{b u} \in L^1$. Despite the fact that computations for $A_\veps$ and $B_\veps$ are trivial, we sketch here a general argument that is helpful in the Wiener setting.

\subsubsection{Integrability of $A_\veps$ and $B_\veps$ via divergence identities}

Let $M$ be any linear transformation of $\rnum^{2d}$ which preserves the Lebesgue measure $\cL^{2d}$. Then, for every field $c: \rnum^{2d} \to \rnum^{2d}$, a distributional identity holds:
\begin{equation}\label{eq-divergence-general} \dive \bra{ c\circ M } \bra{z} =\sqa{\dive \bra{ M c } }\bra{Mz} \comma \end{equation}
which follows immediately from the identity 
\[ \nabla \bra{\phi \circ M^{-1} } =  \bra{ M^{-1}}^* \bra{\nabla \phi } \circ  M^{-1}\]
and the invariance of the measure under $M$ and its inverse. If we take $M\bra{x,y} = \bra{x-s y, y}$ and $c\bra{x,y} = (0,v\bra{x,y})$, we obtain
\begin{equation}\label{eq-divergence-euclidean} \divey{ v\bra{x_s, y}  } = -s \sqa{\divex{v} } \bra{ x_s, y} + \sqa{\divey{v} } \bra{x_s,y} \period\end{equation}
With $v\bra{x,y} = b\bra{x}$ and $s = \veps$, we obtain
\[ \divey{ b\bra{\xe}  } = -\veps \sqa{\dive b } \bra{ x_\veps}  \in L^1 \period \]
From this fact we conclude that
\[ \veps A_\veps \bra{x,y} = \rho\bra{y} \divey{ b\bra{\xe}  } + \ang{ b\bra{\xe}, \nabla \rho \bra{y} }  \in L^1 \period\]
A computation for $B_\veps$ along the same lines is even easier and we omit it.

\subsection{Anisotropic estimate}

We prove that \eqref{anisotropic} above holds true, with
\[ \Lambda_{\rho} \bra{t,x}  = \int \abs{ \divey{ M_t\bra{x}y \rho\bra{y} }} dy \]
and $M_t\bra{x} \abs{Db} \bra{t,x} = Db\bra{t,x}$ is the polar decomposition of $Db = Db_t dt$ with respect to its total variation measure, which is easily seen to be $\abs{Db} = \abs{Db_t} dt$.

We proceed as follows. First, we fix $\veps>0$ and assume $b$ to be smooth, in order to obtain an estimate for $r^\veps$ in terms of $Db$. Then, still keeping $\veps$ fixed, we extend the validity of this estimate to any $BV$ vector field. Finally, we let $\veps \to 0$ and conclude.

For simplicity, but without any loss of generality, we assume that both $\nor{\beta'}_\infty \le 1$ and $\nor{u}_\infty \le 1$. Recall also that, as above, we omit to write as subscripts both $\rho$ and $t$.

\subsubsection{Fix $\veps>0$ and let $b$ be smooth}

In Section \ref{approximate-renormalization}, we obtained that
\[ r^\veps \bra{x}= \int u^\veps \bra{x^\veps} \sqa{ B_\veps \bra{x^\veps, y } -  A_\veps\bra{x^\veps,y} } dy \period \]
Therefore, we estimate
\begin{equation}\label{initial-commutator} \int \abs{ \varphi \beta'\bra{u_\veps} r^\veps } \le \int \abs{\varphi}\bra{x}\abs{ B_\veps \bra{x^\veps, y } -  A_\veps\bra{x^\veps,y} } dx dy \period\end{equation}
After a change of variables $(x,y) \mapsto (x_\veps,y)$, and recalling the explicit expressions for $A_\veps$ and $B_\veps$, we have that the RHS above is equal to
\[  \int \abs{\varphi}\bra{x_\veps}\abs{ \divey{ \frac{ b\bra{x} - b\bra{x_\veps} }{\veps} \rho\bra{y} } } dx dy \period\]
Since $b$ is assumed to be smooth, we have
\[  \frac{  b\bra{x} - b\bra{x_\veps}}{\veps} = - \fint _0^\veps \frac{d}{ds} b\bra{x_s} = \fint_0^\veps Db\bra{x_s} y  \period \]
For brevity, we write here and below $\fint_0^\veps f\bra{s} = \frac{1}{\veps }\int_0^\veps f\bra{s} ds$. Exchanging divergence and integration with respect to $s$,
\begin{equation}
\label{time-integral}
 \divey{\frac{ b\bra{x} - b\bra{x_\veps} }{\veps}  \rho\bra{y} } =  \fint_0^\veps \divey{  Db\bra{x_s} y \rho \bra{y} } \period \end{equation}
Let $c\bra{x,y} =  \rho \bra{y} Db\bra{x} y  = \rho \bra{y} \partial_y b\bra{x} $: by identity \eqref{eq-divergence-euclidean}, we have
\[ \divey{  Db\bra{x_s} y \rho \bra{y} } = -s \sqa{\divex{c}} \bra{x_s,y} + \divey{ c} \bra{x_s, y} \period\]
Since $\divex{c}$ involves further derivatives of $b$, the following identity is crucial:
\[ \divex{ c } \bra{x_s, y} = \rho\bra{y}\sqa{ \partial_y \dive b} \bra{x_s, y} =   - \rho\bra{y} \frac{d}{ds} \dive b \bra{x_s} \]
because it allows to integrate by parts and conclude that the expression in \eqref{time-integral} coincides with
\[ \rho\bra{y} \sqa{ \dive b\bra{x_\veps} -\fint_0^\veps \dive b \bra{x_s} } +\fint_0^\veps \divey{ c} \bra{x_s, y} \period\]

We estimate therefore \eqref{initial-commutator} above with
\begin{equation}
\label{equation-middle-commutator}
 \int \int \abs{\varphi} \bra{x_\veps}  \sqa{ \fint_0^\veps \abs{ \divey{ c } \bra{x_s, y} }  + \abs{ \fint_0^\veps \dive b \bra{x_s}  - \dive b \bra{x} }\rho\bra{y} }dy dx \period \end{equation}

\subsubsection{Keeping $\veps>0$ fixed, we extend the estimate to a general $b$}

We focus on the first term in the sum \eqref{equation-middle-commutator} just above. Exchanging integrations and changing variables $(x,y) \mapsto (x^s,y)$, we find an equivalent expression, of the form
\[  \int \sqa{ \int \abs{\varphi}_\veps \bra{x,y} \Lambda_\rho \bra{x,y} dy } \abs{Db} \bra{x} dx \comma \]
where 
\[ \abs{\varphi}_\veps \bra{x,y} =   \fint_0^\veps \abs{\varphi\bra{x_{\veps-s}}}  \comma \quad \Lambda_\rho  \bra{x,y} = \abs{ \divey { M \bra{x} y \rho \bra{y} } } \]
and $M\bra{x}$ is defined by the identity $M\bra{x} \abs{Db}\bra{x} = Db\bra{x}$. The crucial observation is that this is an expression of the form
\begin{equation}
\label{eq-reshetnyak}
 \int f\bra{x, \frac{ Db }{\abs{Db} } \bra{x} } \abs{Db}\bra{dx} \comma\end{equation}
where $f: \rnum^d \times \mathbb{S}^{d^2-1}$ is continuous and bounded. By Reshetnyak continuity theorem (Theorem 2.39 in \cite{MR1857292}) we extend this estimate to a general $BV$ vector field. More precisely, we approximate $b$ with a sequence of smooth vector fields $b_n$, such that
\[ \lim_{n\to \infty} \nor{b_n - b}_1 = 0, \quad \lim_{n\to \infty} \nor{\dive b_n - \dive b}_1 = 0  \]
and $\abs{Db_n} \bra{ (0,T) \times \rnum^d} \to \abs{Db_n} \bra{ (0,T) \times \rnum^d}$ (such a sequence exists, e.g.\ by convolution with a smooth kernel).

The second term in \eqref{equation-middle-commutator} is easily seen to be continuous with respect to $L^1$ convergence of $\dive b$: we obtained that $\int \abs{ \varphi \beta'\bra{u^\veps} r^\veps}$, for a general $BV$ field is estimated by the sum of two terms:
\begin{equation} \label{final-commutator-1} \int \sqa{ \int \abs{\varphi}_\veps \bra{x,y} \Lambda_\rho \bra{x,y} dy } d \abs{Db}\bra{x} \end{equation}
and
\begin{equation}
\label{final-commutator-2}
\int \int \abs{\varphi}\bra{x_\veps} \abs{ \fint_0^\veps \dive b \bra{x_s}  - \dive b \bra{x} } \rho\bra{y} dy dx \period \end{equation}

\subsubsection{We let $\veps \to 0$}

We consider separately each term that we obtained. In \eqref{final-commutator-2}, we estimate $ \abs{\varphi}\bra{x}$ with its supremum and then use strong continuity in $L^1$ of translations, together with the fact that $\rho$ has compact support, to show that it converges to zero. In \eqref{final-commutator-1}, we exploit the fact that $\varphi$ a test function, so that $\abs{\varphi}_\veps\bra{x,y}$ converges pointwise everywhere to $\abs{\varphi}\bra{x}$ and dominated by some constant: by Lebesgue's theorem with respect to $\abs{Db}$, we obtain \eqref{anisotropic}.

\subsection{Optimization}
We prove that, for every square matrix $M \in \rnum^{d\times d}$, it holds
\[ \inf\cur{\int \abs{ \dive_y \bra{My \rho\bra{y} }} dy  \, : \, \rho \in C^\infty_c (\rnum^d), \rho \ge 0, \int \rho = 1 } = 0\period\]

This is precisely Lemma 5.1 in \cite{MR2408257}: for completeness, we report here the proof, that will be used, almost verbatim, in Section \ref{section-proof} below. We argue that, for any $T>0$, there exists some $\rho \in C^\infty_c \bra{\rnum^d} $ such that
\[ \int \abs{\divey{My \rho\bra{y}}} dy \le \frac 2 T \period\]

Given the vector field $\hat{M}: y \mapsto My$, let $u_0$ be any smooth, non-negative and compactly supported  function with  $\int u_0 =1$. The solution $\bra{u_t}$ of the continuity equation with initial datum $u_0$,
\[ \frac{d} {dt} u_t  + \dive \bra{\hat{M} u_t} = 0 \comma\]
is smooth, non-negative and compactly supported and the same applies to its average $\frac{1}{T}\int_0^T u_t dt =: \rho$. Finally, we estimate
\[ \int \abs{\divey{My \rho\bra{y}}} dy\comma \]
by taking any smooth function $\varphi$ and computing $\int \varphi \bra{y} \divey{My \rho\bra{y} } dy$:
\[ \frac 1 T \int_0^T \int \varphi \bra{y}\divey{\hat{M}\bra{y} u_t\bra{y} } = \frac 1 T \bra{ \int \varphi \bra{y} \bra{u_T - u_0} dy } \le \frac{2} {T} \nor{\varphi}_\infty\period \]

\section{Sobolev and $BV$ spaces on Wiener spaces} \label{section-wiener}


Let us fix some Wiener space $(X, \gamma, H)$: $X$ is a separable Banach space, $\gamma$ is a non-degenerate Gaussian measure on $X$ and $H$ denotes its Cameron-Martin space. Moreover, let $Q \in \cL \bra{X^*, X}$ be the covariance operator associated with $\gamma$.  

\subsection{Smooth maps}
Let us fix an orthonormal basis $\bra{h_n = Qx^*_n}_{n\ge1} \subseteq  H$ induced by elements in the dual space $X^*$ and let $\pi_N: X \to X$ ($N\ge1$) be the \emph{projection} operator, defined on $X$ by
\[ \pi_N\bra{x} = \sum_{n=1}^N \ang{x, x^*_n}h_n \period \]
With a slight abuse of notation, we identify its image with $\rnum^N$.

Given any separable Hilbert space $K$, a map $b: X \to K$ is said to be \emph{cylindrical} if, for some $N \in \nnum$, there exists $k_1, \ldots k_N \in K$ and $b_1,\ldots b_N: \rnum^N \to \rnum$ such that for $\gamma$-a.e.\ $x\in X$, it holds
\[ b\bra{x} = \sum_{i=1}^N  b_i \bra{\pi_N \bra{x} } k_i \period \]
In the case $K = H$, we require moreover that $k_i = h_i$, and use the term \emph{field} to indicate $K$-valued maps. A cylindrical map is said to be \emph{smooth} if there is a representation as above with all the $b_i$'s bounded together with their derivatives.

On smooth cylindrical maps, the Sobolev-Malliavin gradient is defined as
\[ \nabla b \bra{x} = \sum_{i=1}^N  k_i \otimes \bra{Q D_F b_i \bra{x}} \in K \otimes H\comma\]
where $D_F b \bra{x} \in X^*$ is the Fr\'echet derivative of $b_i$ at $x$. The Gaussian divergence of a smooth cylindrical field $b$ (taking values in $H$) is defined by
\[ \dive b \bra{x} = \sum_{i=1}^N \ang{ \nabla b_i\bra{x}, h_i}_H - x^*_i\bra{x} b_i\bra{x} \period\]
For a smooth cylindrical map $b$, taking values in $K\otimes H$, we define its divergence first by representing it as $b = \sum_i k_i \otimes b_i$ and then letting $\dive b = \sum_i k_i \dive b_i$.

It is customary and useful to endow the product spaces $K\otimes K'$ with the Hilbert-Schmidt norm, thus reproducing a Hilbert space structure. 

The following integration by parts formula actually justifies the definition of divergence: for any smooth cylindrical function $\varphi$ and field $b$ it holds
\[ \int \ang{\nabla \varphi, b}_H d\gamma = - \int \varphi\,  \dive b \, d\gamma \period\]
This identity can be generalized to smooth cylindrical $K$-valued maps $\varphi$ and $K\otimes H$-valued maps $b$, arguing componentwise:
\begin{equation}\label{eq-duality-wiener} \int \ang{\nabla \varphi, b}_{K\otimes H} d\gamma = - \int \ang{\varphi ,\dive b}_K d\gamma \period \end{equation}

\subsection{Sobolev spaces and $BV$ maps}

We recall some basic facts about Sobolev-Malliavin spaces, referring to Chapter 5 in \cite{MR1642391} for a detailed description. Given $p \in [1,\infty[$, one considers either the abstract completion  of smooth $K$-valued maps with respect to the norm
\[ \nor{\varphi}_{1,p}  = \nor{ \abs{b}_K + \abs{\nabla b}_{K\otimes H} }_{L^p\bra{X,\gamma}} \comma\]
or the space of maps $\varphi \in L^p\bra{X,\gamma;K}$ such that there exists some $\nabla \varphi \in L^p\bra{X,\gamma;K\otimes H}$ which satisfies \eqref{eq-duality-wiener}, for every smooth map $b$. A well-known result (see e.g.\ Proposition 5.4.6 in \cite{MR1642391}) shows that the two definitions are equivalent: we denote this space by $W^{1,p}\bra{X,\gamma: K}$. In particular, smooth $K$-valued maps are dense.

We briefly introduce $BV$ maps, referring to \cite{MR2558177} for more details: $BV(X,\gamma; K)$ is defined as the space of $u \in L \log^{1/2}L \bra{X,\gamma;K}$ such that there exists some $K\otimes H$-valued measure $Du$ on $X$, with finite total variation, such that for every smooth $K$-valued cylindrical map $\varphi$ it holds
\[  \int \ang{ \varphi, dDu }_{K\otimes H} = -\int u \, \dive \varphi \,   d\gamma \period \]
Theorem 4.1 in \cite{MR2558177} provides the following alternative characterization: $\varphi  \in BV(X,\gamma; K)$ if and only if there exists some sequence $\varphi_n$ of smooth cylindrical maps such that, as $n\to \infty$, $\nor{\varphi_n -\varphi}_1 \to 0$ and $\nor{\nabla \varphi_n}_1$ is uniformly bounded (the smallest bound among all the sequences is exactly the total variation $\abs{D\varphi}(X)$). Actually, there, it is stated in the scalar case, but the argument can be easily extended to deal with general $K$-valued maps.

Using \eqref{eq-duality-wiener}, we introduce the notion of distributional divergence for a map $b$ taking values in $K\otimes H$: in particular we may consider sufficient conditions to ensure that \eqref{eq-duality-wiener} holds true for some function $\dive b \in L^p(X,\gamma)$. If $p \in ]1,\infty[$, it can be proved that $b \in W^{1,p}\bra{X,\gamma;H}$ estails that $\dive b \in L^{p}(X,\gamma)$: the case $p = 2$ is elementary and the others follows essentially from the boundedness of the Riesz transform on Wiener spaces (see Section 5.8 in \cite{MR1642391}). To the author's knowledge, the validity of the correspondent statement for $p=1$ is open.

\section{The continuity equation and statement of the main result} \label{section-main-result}

To give a meaning to the continuity equation \eqref{eq-continuity}, we introduce a suitable class of \emph{test functions}, precisely those  of the form $\varphi\bra{t, \pi_N\bra{x}}$ for some smooth bounded $\varphi\bra{t, y}$, supported in some strip $[\delta, T-\delta] \times \rnum^N$. Then, given a (time-dependent) $H$-valued vector field $b = (b_t)$, we say that $u=\bra{u_t}$ is a distributional solution to \eqref{eq-continuity}, if for every test function $\varphi$ it holds
\[ \int_0^T \int \bra{\frac{d}{dt} \varphi} u_t d\gamma   + \int_0^T \int \ang{\nabla \varphi , b_t }_H u_td\gamma = 0\period\]
As already mentioned in the introduction, for a general overview of the Sobolev case and the links between continuity equations and flows, we completely refer to \cite{MR2475421}. In this work, we only focus on the renormalization property for solutions $u \in L^\infty\bra{(0,T)\times X}$. Precisely, our main result is the following theorem.

\begin{theorem} \label{main-theorem}
Let $b \in L^1 \bra{ (0,T); BV\cap L^p (X, \gamma; H) }$ ($p>1$) with $\dive b \in L^1\bra{(0,T)\times X}$. Any distributional solution $(u_t) \in L^\infty((0,T)\times X)$ of
\[ \frac{d}{dt} u_t + \dive\bra{b_t u_t} = 0 \]
is a renormalized solution.
\end{theorem}

We remark that the assumption $b \in L^p(X,\gamma;H)$ is made for convenience and the proof of the theorem suggests that it may be removed, exploiting the integrability condition $b \in L \log ^{1/2} L(X,\gamma, H)$: however, we prefer here to avoid the introduction of Orlicz spaces.

As already anticipated in the introduction, following the arguments in Section 4 of \cite{MR2475421}, from the result above one obtains existence (under a stronger bound on the divergence), uniqueness and stability of associated Lagrangian flows: we do not enter into details since this is rather plain.

We end this section with a remark that will be useful later on. A general theory of distributions on $(0,T)\times X$ is possible, by considering linear functionals $L$ on test functions such that, for some $k\ge 0$ and $p \in (1,\infty)$ it holds
\[ \abs{ L\bra{\varphi} } \le C(L) \bra{ \nor{ \varphi}_p +\nor{ \nabla \varphi}_p   +\ldots + \nor{ \nabla ^k \varphi}_p } \period\]
In this sense one can consider $ \frac{d}{dt} u_t$ and $\dive\bra{b_t u_t}$ as distributions. However, assume that for some positive measure $\mu$ on $(0,T)\times X$, a given distribution satisfies
\[ \abs{ L\bra{\varphi} } \le \int \abs{\varphi} d\mu \comma \]
for every test function $\varphi$. Then it must coincide with a measure (actually, the restriction to smooth functions of the integration with respect to a measure), which is unique and absolutely continuous with respect to $\mu$. Indeed, it is sufficient to remark that in such a case $L\bra{\varphi}$ defines a continuous functional on $L^1\bra{(0,T)\times X, \mu}$, on a dense subspace (see also Section 2.1 in \cite{MR2558177}).

\section{An example}\label{section-example}

In this section, we show that $BV$ vector fields arise naturally, even in contexts where higher regularity is expected, thus providing a motivation for our result.

The problem is the following: given a Sobolev field $b \in W^{1,2}(X, \gamma; H)$ with $\dive b \in L^\infty(X, \gamma)$, the results in \cite{MR2475421} entail that a global (i.e.\ on all $X$) Lagrangian flow is defined.  However, it is not clear at all how to prove analog \emph{local} results, e.g.\ if $b$ is regular, or even defined, only in an open regular set $\Omega \subseteq X$: a natural strategy is to consider the field $b \chi_\Omega$, which is expected to be $BV$, and apply our result.

Instead of setting up a general theory, we limit the discussion to a specific example, that of a field $b = \nabla \eta$, where $\eta$ solves the following elliptic problem with Neumann boundary conditions:
\begin{equation}
\label{eq-elliptic}
 \left\{ \begin{array}{ll} -\dive \nabla \eta +\lambda \eta = f & \textrm{in $\Omega$,}\\
\frac{\partial \eta}{\partial n} = 0& \textrm{on $\partial \Omega$.} \end{array} \right.\end{equation}

Equations of this kind naturally arise in infinite dimensional settings as Kolmogorov equations associated to stochastic processes reflected at the boundary of $\Omega$, as investigated e.g.\ in the articles \cite{MR2546750}, \cite{MR2841072}, \cite{MR3055217} and \cite{MR2738922}. The motivation to consider the flow driven by the gradient of the solution is then related e.g.~to the study of critical points of $\eta$: however, our aim here is limited to investigate existence and uniqueness and thus the main difficulty is to establish sufficient regularity for $\nabla \eta$ so that the results above apply.

To this aim, we exploit the results about the regularity of solutions to \eqref{eq-elliptic} obtained in \cite{MR2841072} (case $\alpha = 0$) to which we completely refer for more details. We assume therefore that $X$ is a separable Hilbert space, $\Omega$ is convex and is it also the sublevel set of a sufficiently smooth and non-degenerate function (for the precise assumptions, see Hypothesis 1.1 in \cite{MR2841072}, with $\Omega = \mathcal{K}$). Given any separable Hilbert space $K$, it is then possible to introduce a suitable Hilbert space $W^{1,2}(\Omega, \gamma; K)$ of Sobolev differentiable $K$-valued maps, which is well-defined as the closure of the Sobolev-Malliavin operator $\nabla$, on cylindrical smooth $K$-valued maps defined on all $X$, with respect $L^2(\Omega,\gamma; K)$ convergence. The Hilbertian norm in  $W^{1,2}(\Omega, \gamma; K)$ is given by
\[ \nor{f}^2_{W^{1,2}(\Omega,\gamma; K)} = \int_{\Omega} \abs{f}_{K}^2 + \abs{\nabla f}_{K\otimes H}^2 d\gamma.\]
By construction, smooth $K$-valued maps, defined on all $X$, are dense.

Once this space is built, the notion of solution to \eqref{eq-elliptic} is defined exactly as in the finite-dimensional case. Precisely, any $\eta \in W^{1,2}(\Omega, \gamma)$ such that
\begin{equation}\label{eq-variational}
\int_\Omega \sqa{\ang{\nabla \varphi, \nabla \eta} + \lambda \varphi \eta } d\gamma = \int_{\Omega} f \eta \, d\gamma 
\end{equation}
for every $\varphi \in  W^{1,2}(\Omega, \gamma)$, is said to be a solution of \eqref{eq-elliptic}. If $f \in L^2(\Omega,\gamma)$ and $\lambda >0$, existence and uniqueness for the problem are settled by abstract arguments (Lax-Milgram theorem).

The crucial step is to check whether the solution $\eta$ satisfies the boundary condition, at least in the sense of trace of Sobolev functions. Indeed, the assumptions on $\Omega$ entail that it is a set with finite perimeter and a trace operator is well defined:
\[ T_{\partial \Omega}: W^{1,2}(\Omega, \gamma; K) \to L^2(X, \abs{D\chi_{\Omega}}; K),\]
and extends continuously the identity operator on cylindrical smooth maps.

Theorem 3.5 in \cite{MR2841072}, rephrased in the language of sets with finite perimeters, shows that the unique solution $\eta$ to the problem belongs to the space $W^{2,2}(\Omega, \gamma)$, i.e.\ $\nabla \eta \in W^{1,2}(\Omega, \gamma; H)$ and moreover the trace  $T_{\partial \Omega}\nabla \eta$ is orthogonal to the normal at the boundary,
\begin{equation} \label{neumann} \ang{T_{\partial \Omega}\nabla \eta, \sigma } = 0 \quad \textrm{ $\abs{D\chi_{\Omega}}$-a.e.} \end{equation}
where $\sigma$ is provided by the polar decomposition $\sigma \abs{D\chi_{\Omega}} = D\chi _{\Omega}$.

Assuming these facts, we show now that $b := (\nabla\eta) \chi_{\Omega}$ is a well-defined $H$-valued field such that uniqueness of the flow holds: indeed, $b \in BV\cap L^2(X,\gamma; H) $ with $\dive b \in L^1 (X,\gamma)$. The first assertion follows from the general estimate,
\begin{equation}
\label{eq-estimate-bv} \nor{ v\chi_{\Omega} }_{BV(X,\gamma;K)}  \le C \nor{v}_{W^{1,2}(\Omega,\gamma; K)},\end{equation}
valid for any Hilbert space $K$, which is a consequence of the Leibniz rule, easily proved for cylindrical smooth maps,
\[ D( \varphi \chi_\Omega) = \chi_\Omega  \nabla \varphi \gamma + \varphi D \chi_{\Omega},\]
that entails
\[ \nor{ \varphi \chi_\Omega }_{BV(X,\gamma;K)} \le \nor{ \varphi }_{ W^{1,2}(\Omega,\gamma; K)} + \nor{ \varphi}_{L^1(X, \abs{D\chi_\Omega}; K)}.\]
Then, continuity of the trace operator and density of cylindrical smooth maps in $W^{1,2}(\Omega, \gamma; K)$ give \eqref{eq-estimate-bv}.

To show that $\dive b \in L^1 (X,\gamma)$, we take any smooth function $\varphi$ and compute
\[ \int_X \ang{\nabla \varphi, b} d\gamma = \int_\Omega \ang{\nabla \varphi, \nabla \eta} d\gamma = \int_{\Omega} (f -\lambda\eta) \varphi\]
where we used the fact that smooth functions belong to $W^{1,2}(\Omega,\gamma;H)$ and so \eqref{eq-variational} holds. This entails that
\[ \dive b = (\lambda\eta - f)\chi_\Omega \in L^1(X, \gamma).\]

This is sufficient to obtain uniqueness in $L^\infty ((0,T) \times X)$ for the continuity equation associated to $b$ and therefore for Lagrangian flows driven by $b$. To ensure existence, the results in \cite{MR2475421} require  more integrability on $\dive b$, which we obtain, by a comparison principle, under the additional assumption that $f \in L^\infty (\Omega, \gamma)$.

More precisely, if $f \le M \in \rnum$, $\gamma$-a.e.\ in $\Omega$, we conclude that $\lambda \eta \le M$. This follows from  taking $v = \max\cur{ \eta, M} - M$ in \eqref{eq-variational}: arguing by density, it holds $v \in W^{1,2}\bra{\Omega,\gamma}$ with $\nabla v = (\nabla \eta)\chi_{\Omega \cap \cur{\lambda \eta < M }}$. We conclude that
\[ \int _{\Omega \cap \cur{\lambda \eta \ge M} } \abs{\nabla u}^2 \le \int_{\Omega} \bra{f - \lambda \eta} v \le 0\]
which gives  that $\eta \le M$ $\gamma$-a.e. Similarly one shows that $\eta \ge - \nor{f}_\infty$ and so the Lagrangian flow associated to $b$ is well-posed.

\section{Some technical results}\label{section-technical}

Before we address the proof of Theorem \ref{main-theorem}, in this section we prove some auxiliary facts, related to approximations of fields and exponential maps in Wiener spaces. 

\subsection{Cylindrical approximations} We establish two propositions: the first one being a slight generalization of the approximation procedure employed in the proof of Proposition 3.5 in \cite{MR2475421}.

Recall that, in Section \ref{section-wiener}, we introduced an orthonormal basis in $H$ of the form  $\bra{h_n = Q x^*_n}_{n\ge1}$ and related projections operators $\pi_N$: we let in all what follows $\cF_N$ be the $\sigma$-algebra generated by the map $\pi_N$ and let $\eE_N$ be the conditional expectation operator with respect to $\cF_N$.

Moreover, the map $x \mapsto \bra{ \pi_N \bra{x}, x - \pi_N \bra{x} }$ induces decompositions $X = Im \pi_N \oplus Ker \pi_N$ and $H =Im \pi_N \oplus Im\pi_N^{\perp}$. Recall that we tacitly identify $Im \pi_N = \rnum^N$ via  $h_i \mapsto e_i$.  The same map induces a decomposition $\gamma = \gamma_N \otimes \gamma_N^\perp$, where $\gamma_N$ is the standard $N$-dimensional normal law on $\rnum^N$ and $\gamma_N^\perp$ is  a non-degenerate Gaussian measure on $Ker\pi_N$, with Cameron-Martin space given by $ Im\pi_N^{\perp}$.

Given any Hilbert space $K$ and a $K$-valued measure $\mu$ on $X$, its push-forward $\bra{\pi_N}_\sharp \mu$ is defined by
\[ \bra{\tilde{\pi}_N}_\sharp \mu \bra{A} = \mu \bra{\pi_N^{-1} A}\period \]
Notice that push-forwards of $K$-valued measures commute with linear operators on $K$, so that in particular, for any $H\otimes H$-valued measure, it holds \[ \pi_N \otimes \pi_N \mu \bra{A} = \sqa{\pi_N \mu \pi_N } \bra{A} =\pi_N \mu\bra{A} \pi_N \]

In the next propositions, we fix any $b \in BV\cap L^p\bra{X,\gamma; H}$ with $\dive b \in L^q\bra{\gamma}$, $p,q \in [1,\infty[$, and for any $N\ge1$, we let
\[ b^N =  \eE_N\sqa{ \pi_N b} \]
be a cylindrical approximation of $b$.

\begin{proposition}\label{prop-cyl-approx}
Let $b$ and $b^N$ be defined above: then, $b^N$ is a cylindrical $BV$ vector field, with \[D b^N = \sqa{\bra{\pi_N}_\sharp \bra{\pi_N\, Db \,\pi_N} } \otimes \gamma_N^\perp\quad \textrm{and} \quad \dive b^N =\eE_N\sqa{ \dive b}.\]
It holds moreover
\[ \lim_{N\to \infty} \nor{b^N - b}_{p} + \nor{\dive b^N - \dive b}_{q} = 0 \period\]
\end{proposition}

\begin{proof}
Once the identity involving the divergence is proved, the last statement follows at once by the Martingale convergence theorem, or because conditional expectations are contractions and convergence is true for cylindrical fields.

Therefore, we focus on the two identities: being $N$ fixed we omit to write it, so that $\pi = \pi_N$ and $\eE = \eE_N$.

Notice that the field $b^N$ is at least as integrable as $b$, since projections and conditional expectations reduce norms. In what follows, we use duality with smooth cylindrical functions, self-adjointness of $\pi$ and $\eE$, and the commutation relation
\[ \pi \eE \sqa{\nabla \varphi}  = \eE \sqa{\pi \nabla \varphi} = \nabla \eE \sqa{\varphi} \period \]

It holds
\[ \begin{split} \int \varphi \, \dive b^N d\gamma &= - \int \ang{ \nabla \varphi, b^N } d\gamma = - \int \ang{\pi \eE\sqa{\nabla \varphi}, b} d\gamma = \\ &= - \int \ang{ \nabla \eE\sqa{ \varphi }, b } = \int \varphi \, \eE\sqa{ \dive b }\comma \end{split} \]
that gives the identity for the divergence.

Similar computations can be performed on $H\otimes H$ smooth maps with $I\otimes \pi$ ($I$ denotes the identity map) in place of $\pi$. First, we prove that $D \bra{\pi b} = \bra{\pi \otimes I } D b$:
\[ \int \ang{\varphi,d D \pi b } = - \int \ang{ \pi \dive \varphi , b}  =- \int \ang{ \dive\bra{  \pi\otimes I \varphi} , b} =  \int \ang{\varphi, \bra{\pi \otimes I } d D b } \comma \]
and so we conclude that
\[ \begin{split}  \int \ang{\dive \varphi , b^N} & = \int \ang{ \eE_N \sqa{ \dive \varphi }, \pi b } = \int \ang{      \dive\bra{ I\otimes \pi \eE_N\sqa{ \varphi} }, \pi b } =\\ & = \int \ang{  I \otimes \pi \eE_N\sqa{  \varphi }, d D \pi \ b } = \int \ang{ \eE_N\sqa{ \varphi } , \bra{\pi \otimes \pi} d Db }\period \end{split} \]
\end{proof}



The next result is actually mostly measure-theoretical and its proof is based on a disintegration of measures and an application of Jensen's inequality. Notice that it proves and generalize the inequality
\[ \abs{Db^N}\bra{X} \le \abs{Db}\bra{X} \period\]

\begin{proposition}\label{prop-cyl-approx-2}
Let
\[ f: X \times \bra{H \otimes H} \to [0,\infty[ \]
be Borel, positively homogeneous and convex in the second variable, keeping fixed the first. For any $N \ge 1$, it holds
\[ \int f\bra{\pi_N, \frac{Db^N}{\abs{Db^N} }   } d\abs{Db^N} \le \int f\bra{\pi_N , \pi_N \frac{Db}{\abs{Db} }  \pi_N } d\abs{Db} \]
\end{proposition}

\begin{proof}
We omit the dependence upon $N$ and write for brevity $\pi = \pi_N$. We let also $\mu = \pi\otimes \pi Db$, $\nu = Db^N$ and $\rho = \gamma_N^\perp$ so that we rewrite Proposition \ref{prop-cyl-approx} as $\nu = \bra{\pi_\sharp \mu } \otimes \rho$.
The total variation and the polar decomposition of $\nu$ factorize as
\[ \abs{\nu}\bra{dx,dy} = \abs{\pi_\sharp \mu  }\bra{dx} \otimes \rho \bra{dy} \quad \textrm{and} \quad \frac{\nu}{\abs{\nu}}\bra{x,y} = \frac{\pi_\sharp \mu }{\abs{ \pi_\sharp \mu}}\bra{x} \period \]
Therefore, it holds
\[  \int f\bra{x, \frac{\nu}{\abs{\nu} } \bra{x,y}   } d\abs{\nu\bra{x,y}} = \int f \bra{x,\frac{ \pi_\sharp \mu }{\abs{   \pi_\sharp \mu}} \bra{x} } \abs{  \pi_\sharp \mu  }\bra{dx} \period \]
Since $\abs{\pi_\sharp \mu} \le \pi_\sharp\abs{\mu}$, it holds $\frac{\pi_\sharp \mu}{  \pi_\sharp \abs{\mu}} = {\frac{\pi_\sharp \mu} {\abs{ \pi_\sharp \mu} } }  {\frac {\abs{ \pi_\sharp \mu} }{   \pi_\sharp \abs{\mu}} }$ 
which, by positive homogeneity of $f$, gives
\[ \int f \bra{x,\frac{ \pi_\sharp \mu }{\abs{   \pi_\sharp \mu}} \bra{x} } \abs{  \pi_\sharp \mu  }\bra{dx} = \int f \bra{x,\frac{ \pi_\sharp \mu }{   \pi_\sharp \abs{\mu}} \bra{x} }  \pi_\sharp \abs{\mu  }\bra{dx} \]
We now disintegrate $\abs{\mu}$ with respect to $\pi$, and apply Jensen's inequality. More precisely, since $X$ is a separable Banach space, there exists a probability kernel $N\bra{x,y}$ such that, for every bounded Borel map $g\bra{z}$ it holds
\[ \int g\bra{z} d \abs{\mu}\bra{z} = \int \pi_\sharp \abs{\mu} \bra{dx} \int  g\bra{x,y} N\bra{x,dy} \period\]
Moreover, if  $\sigma \abs{\mu} = \mu$ is the polar decomposition, using $g\bra{z} = h\bra{\pi\bra{z}} \sigma\bra{z}$, we obtain the identity
\[ \frac{\pi_\sharp \mu }{ \pi_\sharp \abs{\mu}}\bra{x} =  \int  \sigma\bra{x,y} N\bra{x,dy} \period \]
By Jensen's inequality, 
\[ f \bra{x,\frac{ \pi_\sharp \mu }{   \pi_\sharp \abs{\mu}} \bra{x} }  \le \int f\bra{x, \sigma\bra{x,y}}   N\bra{x,dy}\period\]
Integrating with respect to $\pi_\sharp \abs{\mu}$, the rhs above gives
\[ \int f\bra{\pi\bra{z}, \sigma\bra{z}} d \abs{\mu}\bra{z}  =  \int f\bra{\pi\bra{z}, \frac{\mu}{\abs{\mu}} \bra{z}} d \abs{\mu}\bra{z} \comma  \]
The homogeneity of $f$ and the identities
\[  \frac{\pi Db \pi }{\abs{\pi Db \pi} }\,  \frac{\abs{\pi Db \pi}} { \abs{ Db } } = \frac{\pi Db \pi } {\abs{Db} } = \pi \frac{Db}{\abs{Db} } \pi \comma\]
allow to conclude.
\end{proof}

\subsection{Exponentials maps in Wiener spaces}

It is well known that the Cameron-Martin space $H\subseteq X$ is isomorphic to a subspace  $\mathcal{H} \subset L^2\bra{X,\gamma}$ via
\[ h \mapsto \hat{h} = -\dive h \in L^{2}(X,\gamma)\period\]
Notice that the divergence of a constant field is not zero, because the Gaussian measure is not invariant under translations.

We introduce the following notation: given $b\in L^1\bra{X,\gamma; H}$, we define $\hat{b} \in L^{1}\bra{X\times X, \gamma^2}$ by $\hat{b}\bra{x,y} = -\divey{b\bra{x}} \bra{y}$. This provides an embedding $L^1\bra{X, \gamma; H} \subseteq L^1\bra{X\times X, \gamma^2}$.

In similar direction, given a Hilbert-Schmidt operator $M \in H\otimes H$, we let $\hat{M} = -\dive M \in L^{2}\bra{X,\gamma;H}$. On cylindrical operators of the form $M = \sum m_{ij} h_i \otimes h_j$, it holds
\begin{equation}\label{matrix} \hat{M}\bra{x} = \sum_{i} h_i  \sum_j m_{ij} x_j^*\bra{x}  \period\end{equation}
In particular, it holds $||  \hat{M} || _{2} = \abs{M}$ and $\nabla \hat{M}  = M$. 

In this section, we focus on integrability result for the solution of a continuity equation driven by $\hat{M}$
\begin{equation}\label{continuity-exponential}\frac{d}{dt} u_t + \dive (\hat{M} u_t) = 0, \quad u_0 =1\end{equation}

Since $M$ is regular and integrable, the theory developed in \cite{MR2475421} provides existence, uniqueness and stability for solutions up to a time $T$ which depends on the exponential integrability of $\dive \hat{M}$. Looking for  integral bounds on $u_T$ for any $T>0$ (but fixed), already in the case $M = h_i \otimes h_i$, one finds \[ \dive\hat{ M } = \abs{h_i}^2 - \hat{h}_i^2\comma\] whose negative part is exponentially integrable only up to a factor $\alpha < 1/2$ and so the bound developed in \cite{MR2475421} does not help. The following proposition provides, for every $T > 0$,  $L^p$-bounds for $u_T$ and $\abs{\nabla u_T}$ for some some $p(T)>1$. Although the proof makes explicit use of the exponential form of solutions, the key ingredient is a well-known consequence of the so-called concentration of measure, and we claim (but not prove here) that one could prove results of this kind for rather general $H$-Lipschitz fields.

\begin{proposition}\label{lemma-regularity-exponential}
Let $M \in H\otimes H$. Then, for every $T>0$ there exists some $p(T, \nor{M})>1$ such that \eqref{continuity-exponential} admits a (unique) solution $u \in L^\infty \bra{ (0,T); W^{1,p} \bra{\gamma}}$.
\end{proposition}

\begin{proof}
It is sufficient to assume that $X = \rnum^N$, $\gamma = \gamma_N$ is a standard Gaussian and $M$ is a square matrix, provided we obtain bounds that are independent of $N$, the general case following by cylindrical approximation. 

Recall that the Hilbert-Schmidt norm is stronger than the usual operator norm, that $\nor{AB}_{HS} \le \nor{A}_{HS} \nor{B}$, where $B$ is the usual operator norm and finally that the products of Hilbert-Schmidt operators have finite trace: $Tr\sqa{A^2} \le \nor{A}_{HS}^2$.

Notice also that identity \eqref{matrix} shows that $\hat{M}$ is the linear operator given by matrix multiplication.

We rewrite a linear change of variables in a convenient way (see also Chapter 10 in \cite{MR1439752}). If $C$ is any square matrix in $\rnum^N$, the following identity holds true:
\begin{equation}\label{change-variables} \frac{d(I + C)_\sharp(\gamma_N)}{d\gamma_N} \bra{x + Cx} = \abs{\det\phantom{}_2 \bra{I + C} }^{-1} \exp\sqa{ \dive(Cx) + \abs{Cx}^2/2} \period\end{equation}
where  $\det\phantom{}_2 \bra{I + C}= \det\bra{I+ C} \exp\cur{ - Tr\sqa{C} }$ is the Carleman-Fredholm determinant. As a consequence,
\begin{equation}\label{change-variables-corollary}\int \abs{\det\phantom{}_2 \bra{I + C} } \exp\sqa{ -\dive\bra{Cx} - \abs{Cx}^2/2} d\gamma_N \bra{x} = 1 \end{equation}

In the finite dimensional setting, the unique solution of \eqref{continuity-exponential} is well-known to be $u_t \gamma_N = X(t,\cdot)_\sharp \gamma_N$, where $X(t,x)$ is the classical exponential flow,
\[ X(t,x) = \exp\bra{tM}x = x + E_t x \comma\]
where we write $E_t = \sum_{k=1}^\infty (tM)^k/k!$, because in this form we may apply \eqref{change-variables} and obtain
\begin{equation}\label{eq-exponential-law} u_t(X(t,x)) = \abs{\det\phantom{}_2 \bra{I + E_t }}^{-1} \exp\sqa{ \dive\bra{ E_tx} + \abs{E_t x}^2/2} \period \end{equation}
We compute first the determinant, that gives
\[ \det\phantom{}_2 \bra{I + E_t } = \exp\bra{Tr \sqa{ tM - E_t} } = \exp\bra{Tr\sqa{ (tM)^2 \sum_{k=0}^\infty (tM)^k/(k+2)!  } } \]
and estimate the trace,
\[ \abs{Tr\sqa{ (tM)^2 \sum_{k=0}^\infty (tM)^k/(k+2)!  }} \le t^2 \nor{M}_{HS}^2 \exp\bra{ t \nor{M}_{HS}} \comma\]
so that the determinant is bounded below and above.

We compute then
\begin{equation} \label{eq-ut-p}\int u_t^p = \int \abs{\det\phantom{}_2 \bra{I + E_t }}^{-{p-1}} \exp\sqa{ (p-1)\dive\bra{ E_tx} + (p-1)\abs{E_t x}^2/2}\period \end{equation}

We add and subtract a term $(p-1)^2\abs{E_t x}^2$ inside the exponential, apply Cauchy-Schwartz and \eqref{change-variables-corollary}, so that we see that we need to estimate only
\[ \int \exp\sqa{ 2(p-1)^2\abs{E_t x}^2  + (p-1)\abs{E_t x}^2} d\gamma_N\]
since all the determinant terms that appear are bounded (arguing as above).

Exponential integrability of $(p-1)(2p-1)\abs{E_t x}^2$ follows from the fact that $x \mapsto E_tx$ is Lipschitz, with constant bounded by $t \nor{M}_{HS} \exp\bra{t \nor{M}}$. If $T>0$ is kept fixed, we may consider $p = 1 + \veps$, with $\veps$ so small that Theorem 4.5.7 in \cite{MR1642391} applies providing a right bound, which does not depend on the dimension of the space.

To obtain bounds on the gradient $\nabla u_t$, we notice that \eqref{eq-exponential-law} gives
\[ u_t(y)= \abs{\det\,_2 \bra{I + E_t }}^{-1} \exp\sqa{ \dive E_t( \exp\bra{-tM} y)  + \abs{E_t \exp\bra{-tM} y}^2/2} \period\]
Differentiating with respect to $y$, we obtain
\[  \nabla u_t(y)= u_t(y) \nabla \sqa{ \dive E_t( \exp\bra{-tM} y)  + \abs{E_t \exp\bra{-tM} y}^2/2}\]
Since we already have a bound on $u_t$, it is sufficient to bound the gradient terms, but these are all linear expressions in $y$, which can be explicitly computed and bounded in every $L^p$ space ($p<\infty$) with some constant depending on $p$, $T$ and $\nor{M}$ only. 

\end{proof}

\section{Proof of Theorem \ref{main-theorem}}\label{section-proof}

The line of reasoning mirrors that of the proof of Theorem \ref{theorem-1}, but details are rather different. Therefore, we split this section in steps corresponding to those stated there.

\subsection{Mollification}

We let $\rho$ be any cylindrical smooth function on $X$ and introduce a modified Ornstein-Uhlenbeck, acting only on the space variables,
\[ T^\veps_\rho \varphi \bra{t,x} = \int \varphi\bra{t,\xe} \rho \bra{y} d\gamma\bra{y}\comma \]
where we write, here and in what follows,
\[ \xe = e^{-\veps} x + \sqrt{ 1- e^{-2\veps}} y, \quad \ye =  - \sqrt{ 1- e^{-2\veps}} x + e^{-\veps} y\period\]
Its adjoint (in $L^2((0,T)\times X)$) is given by
\[ \bra{ T^\veps_\rho}^* \varphi_t\bra{x} = \int \varphi\bra{t,x^\veps} \rho\bra{y^\veps} d\gamma\bra{y}\comma \]
where we write, here and in what follows,
\[ x^\veps = e^{-\veps} x - \sqrt{ 1- e^{-2\veps}} y \comma \quad y^\veps = \sqrt{ 1- e^{-2\veps}} x + e^{-\veps} y\period\]
These operators preserve test functions and so we may also define $(T^\veps_\rho)^*$ on distributions, by duality
\[ \ang{\varphi, \bra{T^\veps_{\rho}}^* L} = \ang{T^\veps_{\rho}\varphi, L} \period\]
We let finally $u_\rho^\veps = (T_\rho^\veps)^*u$, so that
it holds
\[ (r_\rho^\veps)_t = \dive\bra{ b (u_\rho^\veps)_t} - (T_\rho^\veps)^*\dive\bra{b u_t}  \period\]

\subsection{Approximate renormalization}\label{approximate-renormalization-wiener}
To keep notation simple, we frequently omit here and below the dependence on $t$ and $\rho$, since they play no role.

As $u \in L^\infty$, an integration by parts shows that $u^\veps$ belongs to every Sobolev space $W^{1,p}(X,\gamma)$ with respect to the space variables, so that the only thing to be proved to justify the usual calculus rules that we perform in this step, is that the distribution $r^\veps$ is (induced by) an integrable function and so it is enough to prove that both $\dive\bra{b u^\veps}$ and $\bra{T^\veps}^*\dive\bra{b u}$ are integrable functions (this is a standard argument, compare with Lemma 3.6 and the computations in Theorem 3.7 in \cite{MR2475421}). We do not enter into details, but to actually perform the computations to get the approximate renormalization we also use the extra integrability assumption $b \in L^p(X,\gamma; H)$.

\subsubsection{Equivalent expressions for $\dive\bra{b u^\veps}$ and $\bra{T^\veps}^*\dive\bra{b u}$ via integration by parts}

Here and in what follows, the term $C_\veps =  e^{\veps} \sqrt{1-e^{-2\veps}}$ frequently appears due to various integration by parts. Notice also that $C_\veps \sim \sqrt{\veps}$ as $\veps \to 0$.

Let $\varphi$ be a test function and compute
\[ T^\veps_\rho \ang{ \nabla  \varphi, b} \bra{x} = -\frac {1}{C_\veps} \int \varphi\bra{x_\veps} \divey{ e^\veps b\bra{x_\veps} \rho\bra{y} } = \int \varphi\bra{x_\veps} A_\veps\bra{x,y} d\gamma\bra{y} \comma\]
and
\[ \ang{ \nabla  T^\veps_\rho \varphi, b} \bra{x} = -\frac {1}{C_\veps} \int \varphi\bra{x_\veps} \divey{ b\bra{x} \rho\bra{y} } d\gamma\bra{y}  =  \int \varphi\bra{x_\veps} B_\veps\bra{x,y}\period\]

We show that $A_\veps$ and $B_\veps$ are integrable and then change variables $\bra{x,y} \mapsto \bra{x^\veps, y^\veps}$ to conclude, as in the finite dimensional case.

\subsubsection{Integrability of $A_\veps$ and $B_\veps$ via divergence identities}

By rotational invariance of the Gaussian measures, an analogue of identity \eqref{eq-divergence-general} holds true in the Wiener spaces, for vector fields $c: X\times X \to H\oplus H$ and rotations $M = M_s$ defined on $X\times X$ (and then on $H\oplus H$) by
\[ M_s\bra{x,y} = \bra{x_s,y_s} = \bra{e^{-s}x +\sqrt{1- e^{-2s}} y,-\sqrt{1- e^{-2s}} x +  e^{-s}y } \period \]
If we specify $c\bra{x,y} = (0, v\bra{x,y})$ we get the analogue of \eqref{eq-divergence-euclidean},
\begin{equation}
\label{eq-divergence-wiener}
 \divey{ v \bra{x_s,y_s }} = \sqrt{1- e^{-2s}} \sqa{\divex{ v  } } \bra{x_s, y_s} +  e^{-s} \sqa{\divey{ v  } } \bra{x_s, y_s} \period\end{equation}
If we specify moreover $v\bra{x,y} = b\bra{x}$ and $s = \veps$ we obtain
\begin{equation}\label{eq-diverngence-wiener-b} \divey{ b \bra{x_\veps}  } =  \sqrt{1- e^{-2\veps}} \dive b \bra{x_\veps} - e^{-s} \hat{b} \bra{x_\veps, y_\veps}   \in L^1(X\times X, \gamma\otimes \gamma)\end{equation}
because of the integrability assumptions on $b$ and its divergence. This shows that $A_\veps$ is integrable, and computations involving $B_\veps$ can be performed along the same lines.

\subsection{Anisotropic estimate}

We prove that \eqref{anisotropic} holds true in the Wiener setting, with
\[ \Lambda_{\rho} \bra{t,x}  = \int_X \abs{ \divey{ \hat{M}_{t,x} \bra{y} \rho\bra{y} }} d\gamma\bra{y}\comma \]
where $M_{t,x} \abs{Db} \bra{t,x} = Db\bra{t,x}$ is the polar decomposition of $Db = Db_t dt$ with respect to its total variation measure, which is  $\abs{Db} = \abs{Db_t} dt$, a finite measure on $(0,T) \times X$. Moreover, $\hat{M}$ denotes the field associated to $M$, defined in Section \ref{section-technical}.

We proceed as in the finite dimensional case. First, we fix $\veps>0$ and assume $b$ to be cylindrical smooth, in order to obtain an estimate for $r^\veps$ in terms of $Db$. Then, keeping $\veps$ fixed, we extend the validity of this estimate to any $BV$ vector field, first cylindrical and then general. Finally, we let $\veps \to 0$ and conclude.

For simplicity, but without any loss of generality, we assume that both $\nor{\beta'}_\infty \le 1$ and $\nor{u}_\infty \le 1$. As above, we omit to write as subscripts both $\rho$ and $t$.

\subsubsection{Let $\veps >0$ be fixed and $b$ be cylindrical smooth}

We perform some computations that give an estimate involving three terms, two of them being error terms (negligible as $\veps \to 0$) and the third providing the result.  Since Sobolev and $BV$ spaces are well-behaved with respect to linear push-forwards, we may safely work in some fixed finite-dimensional Gaussian space $(\rnum^N, \gamma_N)$.

If we write explicitly the expressions obtained in the previous step, we obtain the estimate

\begin{equation} \label{eq-anisotropic-wiener-1} \int \abs{ \varphi \beta'\bra{u_\veps} r^\veps } \le \int \abs{\varphi}\bra{x_\veps}\abs{ \divey{ \frac{   b\bra{x} -  e^{\veps}  b\bra{x_\veps} }{C_\veps}  \rho\bra{y} }}  dx dy \period\end{equation}
We add subtract $b\bra{x_\veps}$ in the difference and we split
\[ \int \abs{\varphi}\bra{x_\veps} \cur{ \frac{e^\veps-1} {C_\veps} \abs{ \divey{ b\bra{x_\veps} \rho\bra{y} }} +\abs{ \divey{ \frac{  b\bra{x_\veps} - b\bra{x} }{ C_\veps}  \rho\bra{y} }} } dx dy \period\]

The first term in the sum above gives the an error term which is smaller than 
\begin{equation} \label{anisotropic-wiener-first-error} \sqrt{\veps} \nor{\varphi}_\infty \sqa{ \nor{b}_1\nor{ \nabla{\rho}}_\infty+ \nor{\dive b}_1\nor{\rho}_\infty } \comma \end{equation} using \eqref{eq-divergence-wiener} and noticing that $C_\veps \le C \sqrt{\veps}$, for $\veps \in (0,1]$ (and $C$ is some absolute constant).

We focus then on what remains, namely the expression
\begin{equation} \label{eq-divergence-2} \int \abs{\varphi}\bra{x_\veps}  \abs{ \divey{ \frac{  b\bra{x_\veps} - b\bra{x} }{ C_\veps}  \rho\bra{y} } } dx dy  \period \end{equation} Since $b$ is cylindrical smooth, write
\[ b\bra{x_\veps} - b\bra{x} = \int_0^\veps \frac{d}{ds} b\bra{x_s} ds = \int_0^\veps Db\bra{x_s}y_s\frac{ds}{C_s}\comma \]
because of the identity $\frac{d}{ds} x_s  = y_s / C_s$. In all what follows, for brevity, we write
\[ \fint_0^\veps f\bra{s} = \frac{1}{ C_\veps} \int_0^\veps f\bra{s} \frac{ds}{C_s} \comma\]
where the notation is justified by the fact that, as $\veps \to 0$,
\begin{equation}
\label{ou-time-average} \frac{1}{C_\veps} \int_0^\veps \frac{ds}{C_s} \to 1 \period \end{equation}

Exchanging divergence and integration, we obtain
\begin{equation}  \label{eq-lebiniz-divergence} \begin{split} \divey{ \frac{ b\bra{x_\veps} -b \bra{x} }{C_\veps}  \rho\bra{y} } &=  \fint_0^\veps \divey{ Db\bra{x_s}y_s  \rho\bra{y} } = \\ &=  \fint_0^\veps \sqa{\divey{ Db\bra{x_s}y_s } \rho\bra{y} + \ang{Db\bra{x_s}y_s, \nabla \rho\bra{y} } } \period\end{split}\end{equation}
Let us consider the first term in the sum above: write $v\bra{x,y} =  Db\bra{x} y = \partial_y b\bra{x}$ and for $s \in (0,\veps)$, apply identity \eqref{eq-divergence-wiener}, to obtain 
\begin{equation} \label{eq-split-divergence}
\divey{ Db\bra{x_s} y_s  } =\sqrt{1-e^{-2s}} \sqa{\divex{ v} }  \bra{x_s, y_s} +  e^{-s} \sqa{\divey{ v  } } \bra{x_s, y_s}  \end{equation}

Since the term $\divex{v}$ above involves further spatial derivatives of $b$, the following identity, which can be obtained  by inspection in coordinates, is crucial:
\[ \divex{ v } \bra{x_s, y_s} = C_s \frac{d}{ds} \dive b \bra{x_s} + \hat{b} \bra{x_s,y_s}\comma \]
where we used the notation $\hat{b}$ introduced in the previous section. This allows to integrate by parts and conclude that
\[ \begin{split}  \fint_0^\veps \sqrt{1-e^{-2s}} \divex{ v } \bra{x_s, y_s} \rho\bra{y}  = \\ \sqa{ e^{-\veps} \dive b \bra{x_\veps} -  \fint_0^\veps \dive b \bra{x_s} e^{-s}} \rho\bra{y}  +  \fint_0^\veps \sqrt{1-e^{-2s}}  \hat{b}\bra{x_s,y_s}\rho\bra{y} \comma  \end{split} \]
where we used the fact that $\frac{d}{ds}\sqrt{1-e^{-2s}} = e^{-s}/C_s$.

Thanks to these computations we separate from \eqref{eq-divergence-2} another error term, smaller than
\[ \nor{\varphi}_\infty \nor{\rho}_\infty \sqa{ \int \abs{  e^{-\veps} \dive b \bra{x_\veps} - \fint_0^\veps \dive b \bra{x_s} e^{-s} } dxdy + \frac{\veps} {2 C_\veps} \nor{b}_1  } \period\]
The integrand above is a linear expression in $\dive b$, which reminds of some averaged Ornstein-Uhlenbeck. By rotational invariance of Gaussian measures and by \eqref{ou-time-average} above, its $L^1$ norm is bounded by some absolute constant, uniformly in $\veps \in (0,1]$. By densities of smooth functions in $L^1$, it defines therefore some a family of continuous operators $R_\veps\bra{\dive b} \bra{x,y}$ and we estimate 
\begin{equation} \label{anisotropic-wiener-second-error} \nor{\varphi}_\infty \nor{\rho}_\infty \sqa{ \nor{R_\veps\bra{\dive b} }_1 + \frac{\veps} {C_\veps} \nor{b}_1  }\period \end{equation}

The following expression contains precisely what remains to be estimated from \eqref{eq-divergence-2}, i.e.~the second term in the second line of \eqref{eq-lebiniz-divergence} and the second term in the RHS of \eqref{eq-split-divergence},
\[ \int \abs{\varphi}\bra{x_\veps} \fint_0^\veps \abs{ e^{-s}\sqa{ \divey{ v  } }\bra{x_s, y_s} \rho\bra{y} + \ang{v \bra{x_s, y_s} , \nabla \rho\bra{y} } } dx dy  \period\]
Once we exchange integration  and perform a change of variables $\bra{x,y} \mapsto \bra{x^s,y^s}$, which maps $x_\veps$ to $x_{\veps-s}$, we rewrite this expression in a way that easily allows an extension to the $BV$ case, namely,
\begin{equation}
\label{eq-divergence-third-term} \int  f \bra{x, \frac{Db}{\abs{Db} }\bra{x} } \abs{Db}\bra{dx}\comma \end{equation}
where
\[ f\bra{x, M }  = \fint_0^\veps \int\abs{\varphi}\bra{x_{\veps-s}} \abs{ e^{-s}\divey{ \hat{ M} \bra{ y }}  \rho\bra{y^s}   + \ang{\hat{M}\bra{y} , \bra{\nabla \rho} \bra{y^s}}  } dy  \comma \]
recalling that $\hat{M}\bra{y} = My$ in the finite dimensional setting.

\subsubsection{Keep $\veps >0$ fixed and extend the estimate to a general $b$}

The expression in \eqref{eq-anisotropic-wiener-1} is smaller than the sum of three terms, namely \eqref{anisotropic-wiener-first-error}, \eqref{anisotropic-wiener-second-error} and \eqref{eq-divergence-third-term}. We extend the validity of this fact to  cylindrical $BV$ fields, and then to the general case.

Under the assumption that $b$ is cylindrical, everything reduces to a computation in $\rnum^N$, so that it is possible to find smooth cylindrical fields $\bra{b_n}$ such that, as $n\to \infty$,
\[ \nor{b_n -b}_1 \to 0 \comma \quad  \nor{\dive b_n - \dive b}_1 \to 0\comma \quad \abs{Db_n}\bra{X} \to \abs{Db}\bra{X} \]
and even $\bra{Db_n}$ weakly-* converge to $Db$ (an approximating sequence extracted from the usual Ornstein-Uhlenbeck mollification provides such a sequence). The left hand side in \eqref{eq-anisotropic-wiener-1}, together with the first and second error terms \eqref{anisotropic-wiener-first-error}, \eqref{anisotropic-wiener-second-error} pass to the limit with respect to this convergence. The only trouble might be caused by \eqref{eq-divergence-third-term}, but the usual Reshetnyak continuity theorem applies (Theorem 2.39 in \cite{MR1857292}).

We now extend the estimate to cover general $BV$ fields. We consider $b^N = \mathbb{E}_N\sqa{\pi_N b }$ and let $N \to \infty$. Again, \eqref{eq-anisotropic-wiener-1}, together with the first and second error terms \eqref{anisotropic-wiener-first-error}, \eqref{anisotropic-wiener-second-error},  pass to the limit because of Proposition \ref{prop-cyl-approx}. To handle the term \eqref{eq-divergence-third-term}, we prove first that for every $N$ large enough so that both $\varphi$ and $\rho$ are $N$-cylindrical, it holds
\[ \int  f \bra{x, \frac{Db^N}{\abs{Db^N} }\bra{x} } d\abs{Db^N}\bra{x} \le \int  f\bra{x, \frac{Db}{\abs{Db} }\bra{x} } d\abs{Db}\bra{x} \period\]
This follows from Proposition \ref{prop-cyl-approx-2}, since by direct inspection, the left hand side above coincides with
\[  \int  f_N\bra{\pi_N \bra{x}, \frac{Db^N}{\abs{Db^N} }\bra{x}  }  \abs{Db^N}\bra{dx} \comma\]
where
\[  f_N\bra{x, M }  = \fint_0^\veps \int\abs{\varphi}\bra{x_{\veps-s}} \abs{ e^{-s}\divey{ My }  \rho\bra{y^s}   + \ang{My, \bra{\nabla \rho} \bra{y^s}}  } d\gamma_N\bra{y}  \comma \]
which is positively homogeneous and convex in the second variable. We get therefore
\[ \int  f \bra{x, \frac{Db^N}{\abs{Db^N} }\bra{x} } \abs{Db^N}\bra{dx} \le  \int  f_N\bra{x, \pi_N \frac{Db}{\abs{Db} }\bra{x}\pi_N } \abs{Db}\bra{dx} \period\]
We recognize that
\[ \pi_N M\pi_N y = \mathbb{E}^y_N \bra{ \pi_N \hat{M}\bra{y} }\comma \]
and so, again by Proposition \ref{prop-cyl-approx}, applied this time to $\hat{M}$, we obtain
\[ \dive_y {\pi_N M\pi_N y}  =\mathbb{E}_N\bra{ \dive_y{ \hat{M}\bra{y}} }\period \]
Combining these identities in the expression for $f_N$ and recalling that $\varphi$ and $\rho$ are $N$-cylindrical we conclude, since the conditional expectation $\eE_N$ is a contraction in $L^1\bra{\gamma\bra{dy}}$.

\subsubsection{We let $\veps \to 0$} The first error term \eqref{anisotropic-wiener-first-error} is infinitesimal, but also the term \ref{anisotropic-wiener-second-error}, because $\nor{R_\veps\bra{\dive b}}_1 \to 0$ when $b$ is smooth and cylindrical, by dominated convergence and \eqref{ou-time-average}. By uniform boundedness of $R_\veps$ in $L^1$ and again by the approximation provided by Proposition \ref{prop-cyl-approx}, this holds also for any field $b$ with $\dive b \in L^1(X,\gamma)$.

The term \eqref{eq-divergence-third-term} converges to
\[ \int \abs{\varphi}\bra{x}\sqa{ \int \abs{\divey{ \hat{M}_x \bra{y}} \rho\bra{y}  + \ang{ \hat{M}_x \bra{y}, \nabla \rho\bra{y}} } d\gamma\bra{y}} \abs{Db}\bra{dx} \comma\]
since the integrand converges everywhere, being $\varphi$ and $\rho$ cylindrical smooth, uniformly bounded by some constant because, for any $p \in ]1,\infty[$, it holds
\[ f\bra{x,M} \le c_p  \nor{\varphi}_\infty  \bra{ \nor{\rho}_p + \nor{\nabla \rho}_p} \abs{M}\period \]
and $\abs{M} \le 1$ as assured by the polar decomposition theorem.

\subsection{Optimization}

We prove that, for any Hilbert-Schmidt operator $M \in H\otimes H$, it holds
\[ \inf_\rho \int \abs{ \dive \bra{ \hat{M} \rho} \bra{y} } d\gamma\bra{y} =  0\]
where the infimum runs along smooth cylindrical functions $\rho$, with $\rho \ge 0$ and $\int \rho = 1$.

The proof goes as in the finite-dimensional case, once we remark that, for fixed $p>1$,
\[ \rho \mapsto \int \abs{ \dive\bra{ \hat{M} \rho} \bra{y} } d\gamma\bra{y} \period \]
is continuous with respect to convergence in $W^{1,p}(X,\gamma)$ and so, by density of smooth functions, the infimum may run along all $\rho \in \bigcup_{p>1}W^{1,p}$, with $\rho \ge 0$, $\int \rho = 1$.

Therefore, for fixed $T>0$, we repeat the same construction as in the finite-dimensional case, with $u_0 = 1$, because Proposition \ref{lemma-regularity-exponential} assures that $\rho = \frac{1}{T} \int_0^T u_t dt \in W^{1,p}(X,\gamma)$ for some $p(T)>1$, which gives that $\rho$ is admissible.

\bibliography{wienerbib}{}

\providecommand{\bysame}{\leavevmode\hbox to3em{\hrulefill}\thinspace}
\providecommand{\MR}{\relax\ifhmode\unskip\space\fi MR }
\providecommand{\MRhref}[2]{%
  \href{http://www.ams.org/mathscinet-getitem?mr=#1}{#2}
}
\providecommand{\href}[2]{#2}
\begin{thebibliography}{AMMP10}

\bibitem[ACFS11]{MR2849722}
Luigi Ambrosio, Gianluca Crippa, Alessio Figalli, and Laura~V. Spinolo,
  \emph{Existence and uniqueness results for the continuity equation and
  applications to the chromatography system}, Nonlinear conservation laws and
  applications, IMA Vol. Math. Appl., vol. 153, Springer, New York, 2011,
  pp.~195--204. \MR{2849722 (2012i:35225)}

\bibitem[AF09]{MR2475421}
Luigi Ambrosio and Alessio Figalli, \emph{On flows associated to {S}obolev
  vector fields in {W}iener spaces: an approach \`a la {D}i{P}erna-{L}ions}, J.
  Funct. Anal. \textbf{256} (2009), no.~1, 179--214. \MR{2475421 (2009k:35019)}

\bibitem[AF11]{MR2847889}
\bysame, \emph{Surface measures and convergence of the {O}rnstein-{U}hlenbeck
  semigroup in {W}iener spaces}, Ann. Fac. Sci. Toulouse Math. (6) \textbf{20}
  (2011), no.~2, 407--438. \MR{2847889 (2012i:28002)}

\bibitem[AFP00]{MR1857292}
Luigi Ambrosio, Nicola Fusco, and Diego Pallara, \emph{Functions of bounded
  variation and free discontinuity problems}, Oxford Mathematical Monographs,
  The Clarendon Press Oxford University Press, New York, 2000. \MR{1857292
  (2003a:49002)}

\bibitem[AFR13]{MR3034584}
Luigi Ambrosio, Alessio Figalli, and Eris Runa, \emph{On sets of finite
  perimeter in {W}iener spaces: reduced boundary and convergence to
  halfspaces}, Atti Accad. Naz. Lincei Cl. Sci. Fis. Mat. Natur. Rend. Lincei
  (9) Mat. Appl. \textbf{24} (2013), no.~1, 111--122. \MR{3034584}

\bibitem[Amb04]{MR2096794}
Luigi Ambrosio, \emph{Transport equation and {C}auchy problem for {$BV$} vector
  fields}, Invent. Math. \textbf{158} (2004), no.~2, 227--260. \MR{2096794
  (2005f:35127)}

\bibitem[Amb08]{MR2408257}
\bysame, \emph{Transport equation and {C}auchy problem for non-smooth vector
  fields}, Calculus of variations and nonlinear partial differential equations,
  Lecture Notes in Math., vol. 1927, Springer, Berlin, 2008, pp.~1--41.
  \MR{2408257 (2010b:35039)}

\bibitem[AMMP10]{MR2558177}
Luigi Ambrosio, Michele Miranda, Jr., Stefania Maniglia, and Diego Pallara,
  \emph{B{V} functions in abstract {W}iener spaces}, J. Funct. Anal.
  \textbf{258} (2010), no.~3, 785--813. \MR{2558177 (2011c:28033)}

\bibitem[BDPT09]{MR2546750}
Viorel Barbu, Giuseppe Da~Prato, and Luciano Tubaro, \emph{Kolmogorov equation
  associated to the stochastic reflection problem on a smooth convex set of a
  {H}ilbert space}, Ann. Probab. \textbf{37} (2009), no.~4, 1427--1458.
  \MR{2546750 (2010m:60264)}

\bibitem[BDPT11]{MR2841072}
\bysame, \emph{Kolmogorov equation associated to the stochastic reflection
  problem on a smooth convex set of a {H}ilbert space {II}}, Ann. Inst. Henri
  Poincar\'e Probab. Stat. \textbf{47} (2011), no.~3, 699--724. \MR{2841072
  (2012i:60165)}

\bibitem[Bog98]{MR1642391}
Vladimir~I. Bogachev, \emph{Gaussian measures}, Mathematical Surveys and
  Monographs, vol.~62, American Mathematical Society, Providence, RI, 1998.
  \MR{1642391 (2000a:60004)}

\bibitem[BPS13]{bogachev-extensions}
V.~I. {Bogachev}, A.~Y. {Pilipenko}, and A.~V. {Shaposhnikov}, \emph{{Sobolev
  functions on infinite-dimensional domains}}, ArXiv e-prints (2013).

\bibitem[CDL08]{MR2369485}
Gianluca Crippa and Camillo De~Lellis, \emph{Estimates and regularity results
  for the {D}i{P}erna-{L}ions flow}, J. Reine Angew. Math. \textbf{616} (2008),
  15--46. \MR{2369485 (2008m:34085)}

\bibitem[Cru83a]{MR724704}
Ana~Bela Cruzeiro, \emph{\'{E}quations diff\'erentielles ordinaires: non
  explosion et mesures quasi-invariantes}, J. Funct. Anal. \textbf{54} (1983),
  no.~2, 193--205. \MR{724704 (85j:34122)}

\bibitem[Cru83b]{MR724705}
\bysame, \emph{\'{E}quations diff\'erentielles sur l'espace de {W}iener et
  formules de {C}ameron-{M}artin non-lin\'eaires}, J. Funct. Anal. \textbf{54}
  (1983), no.~2, 206--227. \MR{724705 (85j:34123)}

\bibitem[Cru84]{MR759104}
\bysame, \emph{Unicit\'e de solutions d'\'equations diff\'erentielles sur
  l'espace de {W}iener}, J. Funct. Anal. \textbf{58} (1984), no.~3, 335--347.
  \MR{759104 (86h:60119)}

\bibitem[DL89]{MR1022305}
R.~J. DiPerna and P.-L. Lions, \emph{Ordinary differential equations, transport
  theory and {S}obolev spaces}, Invent. Math. \textbf{98} (1989), no.~3,
  511--547. \MR{1022305 (90j:34004)}

\bibitem[DPL13]{MR3055217}
Giuseppe Da~Prato and Alessandra Lunardi, \emph{Maximal {$L\sp 2$} regularity
  for {D}irichlet problems in {H}ilbert spaces}, J. Math. Pures Appl. (9)
  \textbf{99} (2013), no.~6, 741--765. \MR{3055217}

\bibitem[FH01]{MR1837539}
Masatoshi Fukushima and Masanori Hino, \emph{On the space of {BV} functions and
  a related stochastic calculus in infinite dimensions}, J. Funct. Anal.
  \textbf{183} (2001), no.~1, 245--268. \MR{1837539 (2002j:60094)}

\bibitem[Fig08]{MR2375067}
Alessio Figalli, \emph{Existence and uniqueness of martingale solutions for
  {SDE}s with rough or degenerate coefficients}, J. Funct. Anal. \textbf{254}
  (2008), no.~1, 109--153. \MR{2375067 (2008k:60135)}

\bibitem[Fuk00]{MR1761369}
Masatoshi Fukushima, \emph{{$BV$} functions and distorted {O}rnstein
  {U}hlenbeck processes over the abstract {W}iener space}, J. Funct. Anal.
  \textbf{174} (2000), no.~1, 227--249. \MR{1761369 (2002e:60123)}

\bibitem[LBL08]{MR2450159}
C.~Le~Bris and P.-L. Lions, \emph{Existence and uniqueness of solutions to
  {F}okker-{P}lanck type equations with irregular coefficients}, Comm. Partial
  Differential Equations \textbf{33} (2008), no.~7-9, 1272--1317. \MR{2450159
  (2009m:35190)}

\bibitem[LL12]{MR2891455}
Huaiqian Li and Dejun Luo, \emph{Quasi-invariant flow generated by
  {S}tratonovich {SDE} with {BV} drift coefficient}, Stoch. Anal. Appl.
  \textbf{30} (2012), no.~2, 258--284. \MR{2891455}

\bibitem[Luo10]{MR2669050}
Dejun Luo, \emph{Well-posedness of {F}okker-{P}lanck type equations on the
  {W}iener space}, Infin. Dimens. Anal. Quantum Probab. Relat. Top. \textbf{13}
  (2010), no.~2, 273--304. \MR{2669050 (2011e:35145)}

\bibitem[Luo13]{MR3016531}
De~Jun Luo, \emph{Fokker-{P}lanck type equations with {S}obolev diffusion
  coefficients and {BV} drift coefficients}, Acta Math. Sin. (Engl. Ser.)
  \textbf{29} (2013), no.~2, 303--314. \MR{3016531}

\bibitem[MWZ05]{MR2199294}
E.~Mayer-Wolf and M.~Zakai, \emph{The divergence of {B}anach space valued
  random variables on {W}iener space}, Probab. Theory Related Fields
  \textbf{132} (2005), no.~2, 291--320. \MR{2199294 (2007e:60039)}

\bibitem[RZZ10]{MR2738922}
Michael R{\"o}ckner, Rongchan Zhu, and Xiangchan Zhu, \emph{B{V} functions in a
  {G}elfand triple and the stochastic reflection problem on a convex sets of a
  {H}ilbert space}, C. R. Math. Acad. Sci. Paris \textbf{348} (2010),
  no.~21-22, 1175--1178. \MR{2738922 (2011i:60117)}

\bibitem[{\"U}st95]{MR1439752}
Ali~S{\"u}leyman {\"U}st{\"u}nel, \emph{An introduction to analysis on {W}iener
  space}, Lecture Notes in Mathematics, vol. 1610, Springer-Verlag, Berlin,
  1995. \MR{1439752 (98d:60109)}

\end{thebibliography}
\bibliographystyle{amsalpha}

\end{document}